\documentclass[preprint,sort&compress,review,12pt]{elsarticle}

\usepackage{amssymb}

\journal{Information Sciences}
\usepackage{amsmath}
\usepackage{subfigure} 
\usepackage{graphicx}

\usepackage{amsthm}

\usepackage{algorithm}
\usepackage{algorithmic}
\usepackage{float}
\usepackage{lipsum}
\newtheorem{theorem}{Theorem}
\newtheorem{lem}{Lemma}
\newdefinition{definition}{Definition}
\newdefinition{ass}{Assumption}
\newdefinition{rem}{Remark}

\begin{document}

\begin{frontmatter}



\title{One-Point Residual Feedback Algorithms for Distributed Online Convex and Non-convex Optimization}


\author[1]{Yaowen Wang}
\ead{wangyaowen0619@163.com}

\author[2,3]{Lipo Mo\corref{cor1}}
\ead{beihangmlp@126.com}

\author[2]{Min Zuo}
\ead{zuomin@btbu.edu.cn}

\author[4]{Yuanshi Zheng}
\ead{zhengyuanshi2005@163.com}

\address[1]{School of Mathematics and Statistics, Beijing Technology and Business University, Beijing 100048, China}
\address[2]{Institute of Systems Science, Beijing Wuzi University, Beijing 101126, China}
\address[3]{School of Computer and Artificial Intelligence, Beijing Technology and Business University, Beijing 100048, China}
\address[4]{Center for Complex Systems, School of Mechano-electronic Engineering, Xidian University, Xi'an 710071, China}

\cortext[cor1]{Corresponding author}
\fntext[fn1]{This work is supported by the National Natural Science Foundation of China (62473009).}

\begin{abstract}
This paper mainly addresses the distributed online optimization problem where the local objective functions are assumed to be convex or non-convex. First, the distributed algorithms are proposed for the convex and non-convex situations, where the one-point residual feedback technology is introduced to estimate gradient of local objective functions. Then the regret bounds of the proposed algorithms are derived respectively under the assumption that the local objective functions are Lipschitz or smooth, which implies that the regrets are sublinear. Finally, we give two numerical examples of distributed convex optimization and distributed resources allocation problem to illustrate the effectiveness of the proposed algorithm.
\end{abstract}



\end{frontmatter}


\section{Introduction}
Online optimization problem has attracted a significant amount of attention in recent years, where the objective functions are time-varying and unknown to decision maker before selecting a decision \cite{ref1}. Due to the lack of prior knowledge of the objective function, traditional offline optimization algorithms are not applicable \cite{ref35,ref36}. Then, a lot of online optimization algorithms are proposed to deal with this difficulty \cite{ref37,ref38}. For example, Zeroth-order (ZO) optimization algorithm was proposed when the first or second order information of objective function was not accessible \cite{ref16}. Usually, two-point and one-point estimator are utilized in ZO methods. As a type of ZO method, two-point estimator is one of the most direct and effective method and it has been extensively studied in \cite{ref16,ref12,ref13,ref14,ref15,ref17,ref18}, where two distinct points at each time instant are used to estimate the unknown gradient. Using two-point estimator to estimate the gradient can improve the convergence speed and it has low variance. However, two-point estimator is only applicable in scenarios where the same objective function can be accessed multiple times. When the objective function sequence is non-stationary, the two-point estimator is no longer applicable. To overcome this issue, one-point estimator was proposed in \cite{ref20}, where it required the objective function $f_t(x)$ only once at each time. Then, one-point estimator was extended to the situations that the objective functions were smooth, self-concordant regularized, stochastic and non-stationary \cite{ref21,ref22,ref23,ref24,ref32}. In fact, one-point estimator has large variance since it utilized only a small amount of information.\\
\indent To improve the performance of one-point estimator, a one-point residual feedback (ORF) estimator was proposed in \cite{ref34} and \cite{ref25} and it was proved that ORF estimator performs better than conventional one-point estimator and has smaller variance. In the aforementioned studies, all algorithms  proposed are centralized essentially. However, due to high computational demands and difficulties in information transmission in some practical problems such as resources allocation \cite{ref2,ref3}, centralized online optimization algorithms are not effective for some practical peoblems.\\
\indent To deal with the limitation of information transmission and compute ability and find a better approach of tracking the optimal decision sequence, many scholars turned their research to design distributed online optimization algorithms \cite{ref4,ref5,ref6,ref7,ref8,ref31,ref33}. In distributed online optimization problems, each agent in the multi-agent system independently makes its own decisions by communicating with others to find optimal decision sequence. However, when the objective function sequence is non-stationary, the above distributed online optimization algorithms may be not applicable, which is an open problem.\\
\indent In this paper, we extend the ORF estimator to distributed online optimization problems where the objective function sequence is non-stationary. Two distributed algorithms are proposed for online convex and non-convex optimization problems. And the regret bounds of proposed algorithms are derived under the assumptions that the communication graph is undirected and the adjacency matrix is double-stochastic. The primary contributions of this paper are detailed as follows:

\indent\textbf{1.} Compared with \cite{ref34,ref25}, where the optimization algorithms are centralized, the algorithms proposed in this paper are  distributed, which can achieve sublinear regret bounds and have lower variance compared with one-point feedback estimator. 

\indent\textbf{2.} Compared with \cite{ref4,ref5,ref6,ref7,ref8,ref31,ref33}, where the objective functions were assumed to be stationary, this paper aims to solve non-stationary situation. 

\indent\textbf{3.} Compared with conventional one-point feedback algorithms \cite{ref20,ref21,ref22}, where the objective functions are assumed to be uniformly bound, we remove this assumption in this paper and only assume that the variation of the objective function is bounded.

The organization of this article is as follows. In Section 2, we give some preliminaries such as assumptions and lemmas. Section 3 introduces the distributed constrained online convex optimization problems. We give the distributed ORF algorithm for online convex optimization and analyse its regret bound. Similar to Section 3, Section 4 studies distributed ORF algorithm for online non-convex optimization and shows that our algorithm can achieve sublinear regret bound. Finally, Section 5 shows numerical examples with a comparison to traditional one point estimator and prove that our algorithms have lower regret bound and variance.

\indent\textbf{Notations: }$\mathbb{R}^d$ denotes the $d$-dimensional real number space. $\nabla f_t(.)$ denotes the gradient of function $f_t$. $\|x\|$ is the Euclidean norm of a vector $x$ and a vector $x$ is considered as a column. $x^\prime$ denotes the transpose of vector $x$. $\Pi_{\mathcal{X}}$ is the projection operator that project a vector onto a convex set $\mathcal{X}$.

\section{Problem Formulation and Preliminaries}
\subsection{Problem Formulation}
Considering the following distributed online optimization problem 
\begin{equation}\label{1}
  \min_{x\in \mathcal{X}}\sum_{i=1}^{N}\sum_{t=1}^{T}f_t^i(x),
\end{equation}

\noindent where the constraint set $\mathcal{X}\subset\mathbb{R}^{d}$ is a convex set and $f_t^i:\mathbb{R}^d\rightarrow\mathbb{R}$ is local objective function. It is assumed that the local objective function is convex or non-convex. The agent $i$ makes a decision $x_t^i$ and then observation value of local objective function $f_t^i(x_t^i), i=1,2,\ldots,N$ is revealed at any time $t$. Furthermore, the sequence of objective function $\{f_t^i\}$ is assumed to be non-stationary, which means that the objective function $f_t^i$ is changed over dynamic environment and decision sequence of each agent $i=1,2,\ldots,N$ and their neighbor agents. Non-stationary challenge is emerged in machine learning problems.  Such as reinforcement learning of multi-agent systems, agents encounter non-stationary challenges where the environment fluctuates due to both natural noise and adversarial actions by competing entities.. The goal of each agent $i$ is to find an online decision sequence $\{x_t^i\}$ to make the value of global objective function is as close as possible to optimal decision sequence. Each local objective function $f_t^i$ is assumed to be Lipschitz-continuous or smooth, and these two kinds of function are defined as follows.
\begin{definition}(see \cite{ref16})
  The function $f$ is said to be Lipschitz-continuous noted by $f\in C^{0,0}$ if $|f(x)-f(y)|\leq L_0\| x-y\|,\forall x,y\in \mathcal{X}$, where $L_0>0$ is Lipschitz parameter. The function $f$ is said to be smooth noted by $f\in C^{1,1}$ if $|\nabla f(x)-\nabla f(y)|\leq L_1\| x-y\|,\forall x,y\in \mathcal{X}$, where $L_1>0$ is smoothness parameter.
\end{definition}

In order to find the global optimal policy of online optimization problem (\ref{1}), each agent $i$ communicates with its neighbor agents and makes the decision sequence $\{x_t^i\}$. Suppose there are $N$ agents and the communication topology can be described as an undirected graph $\mathcal{G}=(\mathcal{V},\mathcal{E},\mathcal{A})$, where the vertex set $\mathcal{V}=\{1,2,\ldots,N\}$ and the edge set $\mathcal{E}\subset \mathcal{V}\times\mathcal{V}$. The adjacency matrix is denoted as $\mathcal{A}=[a_{ij}]$, where $a_{ij}>0$ if $(i,j)\in\mathcal{E}$ and $a_{ii}>0$ for all $i\in\mathcal{V}$. Then we adopt the following assumption on the adjacency matrix $\mathcal{A}$.

\begin{ass}
  For all $i,j=1,2,\ldots,N$, there exists a constant $0<\epsilon<1$ satisfying:

\noindent(a) $a_{ij}>\epsilon$ if $(j,i)\in \mathcal{E}$.

\noindent(b) $\sum_{i=1}^{N}a_{ij}=\sum_{j=1}^{N}a_{ij}=1$.
\end{ass}

Based on above assumption, we give the following lemma of state transition matrix.
\begin{lem}(see \cite{ref30})
If Assumption 1 holds, it satisfies that
$$\big|[\Phi(t)]_{ij}-\frac{1}{N}\big|\leq\gamma^{t-1}$$
for any $i,j\in\mathcal{V}$ and $t\geq0$, where $\gamma=1-\frac{\epsilon}{4N^2}$ and $\Phi(t)=A^t$ is state transition matrix.
\end{lem}

Moreover, we give the following basic assumption about the constraint set.
\begin{ass}
 There exist positive constants $r_l$ and $r_u$ such that $r_l \mathbb{B}^d\subseteq\mathcal{X}\subseteq r_u\mathbb{B}^d$, where $\mathbb{B}^d$ is a unit ball in $\mathbb{R}^d$.
\end{ass}
\subsection{Preliminaries}
In online optimization problem, it is difficult to find the optimal decision sequence of global objective function since the objective function is unknown before making a decision. To evaluate the performance of online optimization algorithm, it always use static regret which is described as the gap between the decision sequence found by the algorithm and the optimal decision sequence, and it is defined as:
$$R^T=\sum_{t=1}^{T}\sum_{i=1}^{N}f_t^i(x_t^i)-\min_{x\in\mathcal{X}}\{\sum_{t=1}^{T}\sum_{i=1}^{N}f_t^i(x)\}.$$

To solve distributed online optimization problem (\ref{1}), we introduce ZO method since the derivatives of all local objective functions may be not available. The core idea of ZO method is to estimate the gradient of local objective function using the value of objective function. In fact, the smoothed version $f_{\delta,t}^i$ of objective function $f_t^i$ is used in ZO method, where $f_{\delta,t}^i(x)=E_{u_t^i\sim\mathbb{US}^d}[f_t^i(x+\delta u_t^i)]$ and $u_t^i$ is a random vector uniformly sampled in a unit sphere $\mathbb{S}^d$. From the definition, it is clear that original function $f_t^i(x)$ has to be defined over a larger set $\mathcal{X}_\delta=\{z|z=x+\delta v,\text{ for any } x\in\mathcal{X} \text{ and }v\in\mathbb{S}^d\}$ since the iteration point may evaluate outside the constraint set $\mathcal{X}$. Then, we have the following result about property of smoothed function and the approximation errors between smoothed version and its original function.

\begin{lem}(see \cite{ref16})
  The error of function $f_t^i$ and its smoothed function $f_{\delta,t}^i$ satisfies
\begin{equation}
    |f_{\delta,t}^i(x)-f_t^i(x)|=\begin{cases}\delta L_0,&\mbox{if } f_t^i\in C^{0,0}\\\delta^2L_1,&\mbox{if } f_t^i\in C^{1,1},\end{cases} \notag
\end{equation}
and $\|\nabla f_{\delta,t}^i(x)-\nabla f_t^i(x)\|\leq\delta L_1d$, if $f_t^i\in C^{1,1}$, where $L_0$ and $L_1$ are positive constants and represent the Lipschitz and smoothness parameter respectively.
\end{lem}

\begin{lem}(see \cite{ref16})
  If $f_t^i(x)\in C^{0,0}$ is $L_0$-Lipschitz and $x\in\mathbb{R}^d$, then $f_{\delta,t}^i(x) \in C^{1,1}$ is $L_\delta$-Lipschitz with $L_{\delta}=\frac{d}{\delta} L_0.$
\end{lem}

Since the sequence of objective function is natural or adversarial non-stationary, the form of function $f_t^i$ may be different when the agent $i$ takes different decision $x_t^i$ and $x_t^i+\delta u_t^i$ and it can not obtain two different point $f_t^i(x_t^i)$ and $f_t^i(x_t^i+\delta u_t^i)$ at the same time $t$. To tackle this limitation, we use ORF estimator defined as:
\begin{align}{\label{2}}
  \widetilde{g}_{t}^i(x_{t}^i):=\frac{d}{\delta}\big(f_{t}^i(x_{t}^i+\delta u_{t}^i)-f_{t-1}^i(x_{t-1}^i+\delta u_{t-1}^i)\big)u_{t}^i,
\end{align}
where $u_{t-1}^i,u_{t}^i$ are independent random vectors and uniformly sampled in a unit sphere $\mathbb{S}^d$. In the following lemma, we give the basic properties of ORF estimator.

\begin{lem}
The ORF estimator (\ref{2}) satisfies
$\mathbb{E}\left[\widetilde{g}_{t}^i(x_{t}^i)\right]=\nabla f_{\delta,t}^i(x_{t}^i)$ for all $x_t^i\in\mathcal{X}$ and $t.$
\end{lem}
\begin{proof}
  According to \cite{ref20}, we can obtain that $\frac d\delta f_t^i(x_t^i+\delta u_t^i)u_t^i$ is an unbiased estimator of $\nabla f_{\delta,t}^i(x_t^i)$. Then we can conclude the result since the expectation of $u_t^i$ is $0$ and $u_t^i$ is independent from $u_{t-1}^i,x_{t-1}^i.$ 
\end{proof}

Lemma 4 shows that estimator (\ref{2}) is an unbiased estimation of the gradient of smoothed function $f_{\delta,t}^i$. Then, we can give following lemma to bound the second moment of (\ref{2}).

\begin{lem}
If $f_t^i\in C^{0,0}$ with Lipschitz constant $L_0$ for all time $t.$ and the gradient estimator is updated by rule (5), the flowing inequality holds.
\begin{equation}
  \mathbb{E}[\|\widetilde{g}_t^i(x_t^i)\|^2]\leq\frac{3d^2L_0^2}{\delta^2}\|x_t^i-x_{t-1}^i\|^2+12d^2L_0^2+\frac{3d^2}{\delta^2}\theta_{i,t}^2 ,
  \label{3}
\end{equation}
\noindent where $\theta_{i,t}=\sup\limits_{x\in\mathcal{X}_\delta,t=1,2,\ldots,T}\big|f_{t}^i(x)-f_{t-1}^i(x)\big|$ is the increasing rate of objective function $f_t^i$.
\end{lem}
\begin{proof}
From the definition of ORF estimator (2) and the inequality that $(a+b+c)^2\leq3(a^2+b^2+c^2)$, we have that
$$\begin{aligned}\mathbb{E}[\|\widetilde{g}_t^i(x_t^i)\|^2]\leq&\frac{3d^2}{\delta^2}\mathbb{E}\big[\big(f_t^i(x_{t}^i+\delta u_{t}^i)-f_{t}^i(x_{t-1}^i+\delta u_{t}^i)\big)^2\|u_t^i\|^2\big]\\
&+\frac{3d^2}{\delta^2}\mathbb{E}\big[\big(f_t^i(x_{t-1}^i+\delta u_{t}^i)-f_{t}^i(x_{t-1}^i+\delta u_{t-1}^i)\big)\|u_t^i\|^2\big]\\
&+\frac{3d^2}{\delta^2}\mathbb{E}\big[\big(f_t^i(x_{t-1}^i+\delta u_{t-1}^i)-f_{t-1}^i(x_{t-1}^i+\delta u_{t-1}^i)\big)^2\|u_t^i\|^2\big]\\
\leq&\frac{3d^2L_0^2}{\delta^2}\mathbb{E}\big[\|x_t^i-x_{t-1}^i\|^2\big]+12d^2L_0^2+\frac{3d^2}{\delta^2}\theta_{i,t}^2\end{aligned}$$
where the last inequality is based on the conclusion that $\|u_t^i\|^2\leq1$ and $\|u_t^i-u_{t-1}^i\|^2\|u_t^i\|^2\leq2\|u_t^i\|^4+2\|u_t^i\|^2\|u_{t-1}^i\|^2\leq4$.
\end{proof}
\section{Distributed ORF Algorithm for Convex Online Optimization}
In this section, we analyse the distributed online optimization problem where the local objective function is assumed to be convex. Here, we propose the following ORF update rule:
\begin{equation}{\label{4}}
  x_{t+1}^i=\Pi _{\mathcal{X}}\big[\sum_{j=1}^{N}a_{ij}\big(x_t^j-\eta\widetilde{g}_t^j(x_t^j)\big)\big].
\end{equation}
where $\Pi _{\mathcal{X}}$ is projection operator and $\eta>0$ is stepsize. In summary, we give the following algorithm.
\begin{algorithm}
\caption{Distributed ORF Algorithm for Online Convex Optimization}
{\textbf{Initialization}:}
{Initial values of $x^1_0, x^2_0,$
    $ \cdots, x^N_0$, number of iterations T,
    and appropriate value of $\eta$ and $\delta$.}

\textbf{For }$t=0$ to $T$, $i=1$ to $N$ 
   
   \qquad Let $u_t^i$ uniformly sampled in $\mathbb{S}^{d}$ and compute the ORF estimator $\widetilde{g}_{t}^i(x_{t}^i)$ by (\ref{2})
   
   \qquad Update $x^i_{t+1}$ for all agents $i$ by (\ref{4}).\\
\textbf{end for}
\end{algorithm}

Let the  $x^*=\arg\min_{x\in\mathcal{X}}\sum_{t=1}^{T}\sum_{i=1}^{n}f_t^i(x)$ be the optimal decision point, then the static regret can be written as
\begin{equation}\label{5}
  R^T_g=\mathbb{E}\Big[\sum_{t=1}^{T}\sum_{i=1}^{N}\big[f_t^i(x_t^i)-f_t^i(x^*)\big]\Big].
\end{equation}

Based on the definition of regret above, we give the regret bound of Algorithm 1 for two situations with Lipschitz-continuous and smooth objective function.

\begin{theorem}
Suppose Assumption 1 and 2 hold. If $f_{t}^i\in C^{0,0}$ with Lipschitz constant $L_0$ for all $t$ and $i$, run Algorithm 1 with $\eta=\frac{1}{\sqrt{3\alpha}dL_0 T^{\frac{2}{3}}}$ and $\delta=\frac{2}{T^{\frac{1}{3}}}$, the regret bound satisfies
$$R_{g}^T\leq\mathcal{O}\Big(\max\big\{T^{\frac{2}{3}},\Theta_T^2\big\}\Big),$$
where $\Theta_T=\sum_{t=1}^{T}\sum_{i=1}^{N}\theta_{i,t}$ is accumulated increasing rate of objective functions and $\theta_{i,t}$ is defined in Lemma 5.
\end{theorem}
\begin{proof}
Based on the definition of regret (\ref{5}) and Lemma 2, we can obtain that
\begin{align}\label{6}
\nonumber R_{g}^T=&\mathbb{E}\Big[\sum_{t=1}^{T}\sum_{i=1}^{N}\big[f_t^i(x_t^i)-f_t^i(x^*)\big]\Big]
\nonumber\\=&\mathbb{E}\Big[\sum_{t=1}^{T}\sum_{i=1}^{N}\big[f_t^i(x_t^i)-f_{\delta,t}^i(x_t^i)+f_{\delta,t}^i(x^*)-f_{t}^i(x^*)+f_{\delta,t}^i(x_t^i)-f_{\delta,t}^i(x^*)\big]\Big]
\nonumber\\\leq&\mathbb{E}\Big[\sum_{t=1}^{T}\sum_{i=1}^{N}[f_{\delta,t}^i(x_t^i)-f_{\delta,t}^i(x^*)\big]\Big]+2\delta L_0NT.\end{align}
Then, we analyse the first term of above inequality. It follows from the convexity of $f_{t}^i(x)$ that $f_{\delta,t}^i(x)$ is convex for all $t$ and $i$, hence, it satisfies for any $x\in\mathcal{X}$ that $$f_{\delta,t}^i(x_{t}^i)-f_{\delta,t}^i(x)\leq\langle\nabla f_{\delta,t}^i(x_{t}^i),x_{t}^i-x\rangle.$$
According to Lemma 4, we can take expectation over $u_t^i$ and substitute $\nabla f_{\delta,t}^i(x_t^i)$ with $\tilde{g}_t^i(x_t^i)$. The following inequality still holds: 
\begin{align}\label{7}
\mathbb{E}\big[f_{\delta,t}^i(x_t^i)-f_{\delta,t}^i(x)\big]\leq\mathbb{E}\big[\langle\tilde{g}_t^i(x_t^i),x_t^i-x\rangle\big].
\end{align}
To give the bound of right side of (\ref{7}), we can obtain by (\ref{4}) that
$$\begin{aligned}&\sum_{i=1}^{N}\|x_{t+1}^i-x\|^{2}
\\=&\sum_{i=1}^{N}\|\Pi_{\mathcal{X}}\big[\sum_{j=1}^{N}a_{ij}(x_t^j-\eta\tilde{g}_t^j(x_t^j))\big]-\Pi_{\mathcal{X}}\big[x\big]\|^{2}
\\\leq&\sum_{i=1}^{N}\sum_{j=1}^{N}a_{ij}\|x_{t}^j-\eta\tilde{g}_t^j(x_{t}^j)-x\|^{2}
\\=&\sum_{j=1}^{N}\|x_{t}^j-x\|^{2}+\sum_{j=1}^{N}\eta^{2}\|\tilde{g}_{t}^j(x_{t}^j)\|^{2}-2\sum_{j=1}^{N}\eta\langle\tilde{g}_{t}^j(x_{t}^j),x_{t}^j-x\rangle,\end{aligned}$$
where the last equation holds since the adjacency matrix is double-stochastic. Rearranging the above inequality and taking summation over $t=1,2,\ldots,T$, we have that
$$\begin{aligned}&\sum_{t=1}^{T}\sum_{i=1}^{N}\langle\tilde{g}_{t}^i(x_{t}^i),x_{t}^i-x\rangle
\\\leq&\frac{1}{2\eta}\sum_{t=1}^{T}\big[\sum_{i=1}^{N}\|x_{t}^i-x\|^{2}-\sum_{i=1}^{N}\|x_{t+1}^i-x\|^{2}\big]+\frac{\eta}{2}\sum_{t=1}^{T}\sum_{i=1}^{N}\|\tilde{g}_{t}^i(x_{t}^i)\|^{2}
\\\leq&\frac{1}{2\eta}\sum_{i=1}^{N}\|x_{1}^i-x\|^{2}+\frac{\eta}{2}\sum_{t=1}^{T}\sum_{i=1}^{N}\|\tilde{g}_{t}^i(x_{t}^i)\|^{2}.
\end{aligned}$$
Combining above inequality and (\ref{7}), take expectation over $u_t^i$, we have that
\begin{align}\label{8}
\nonumber&\sum_{t=1}^{T}\sum_{i=1}^{N}\mathbb{E}\big[f_{\delta,t}^i(x_t^i)-f_{\delta,t}^i(x)\big]\\
\leq&\frac{1}{2\eta}\sum_{i=1}^{N}\|x_{1 }^i-x\|^{2}+\frac{\eta}{2}\sum_{t=1}^{T}\sum_{i=1}^{N}\mathbb{E}\big[\|\tilde{g}_{t}^i(x_{t}^i)\|^{2}\big].\end{align}

Then, we give the bound of the last term of (\ref{8}). Sum up (\ref{3}) on both sides over $i=1,2,\ldots,N$ and $t=1,2,\ldots,T$, it yields that
\begin{align}\label{9}
\nonumber&\sum_{t=1}^{T}\sum_{i=1}^{N}\mathbb{E}[\|\widetilde{g}_t^i(x_t^i)\|^2]\\
\leq&\frac{3d^2L_0^2}{\delta^2}\sum_{t=1}^{T}\sum_{i=1}^{N}\mathbb{E}\big[\|x_t^i-x_{t-1}^i\|^2\big]+12d^2L_0^2NT+\frac{3d^2}{\delta^2}\Theta_T^2.\end{align}
To bound the above inequality, we have to bound the term $\sum_{t=1}^{T}\sum_{i=1}^{N}\mathbb{E}\big[\|x_t^i-x_{t-1}^i\|^2\big]$. From (\ref{4}), it holds that
\begin{align}
\nonumber&\sum_{t=0}^{T-1}\sum_{i=1}^{N}\mathbb{E}[\|x_{t+1}^i-x_t^i\|^2]\\
\nonumber\leq&\sum_{t=0}^{T-1}\sum_{i=1}^{N}\mathbb{E}[\|\sum_{j=1}^{n}a_{ij}\big(x_t^j-\eta\widetilde{g}_t^j(x_t^j)\big)-x_t^i\|^2]\\
\nonumber\leq&2\sum_{t=0}^{T-1}\sum_{i=1}^{N}\eta^2\mathbb{E}[\|\widetilde{g}_t^i(x_t^i)\|^2]+2\sum_{t=0}^{T-1}\sum_{i=1}^{N}\mathbb{E}[\|\sum_{j=1}^{N}a_{ij}x_t^j-x_t^i\|^2]\\
\nonumber\leq&2\sum_{t=0}^{T-1}\sum_{i=1}^{N}\eta^2\mathbb{E}[\|\widetilde{g}_t^i(x_t^i)\|^2]+2\sum_{t=0}^{T-1}\sum_{i=1}^{N}\sum_{j=1}^{N}a_{ij}\|x_t^j-x_t^i\|^2.
\end{align}
For the last term of the above inequality, we have that
$$\begin{aligned}
\sum_{t=0}^{T-1}\sum_{i=1}^N\sum_{j=1}^{N}a_{ij}\|x_t^j-x_t^i\|^2\leq2\sum_{t=0}^{T-1}\sum_{i=1}^N\sum_{j=1}^{N}a_{ij}\|x_t^i-\bar{x}_t\|^2=2\sum_{t=0}^{T-1}\sum_{i=1}^N\|x_t^i-\bar{x}_t\|^2.
\end{aligned}$$
where $\bar{x}_t=\sum_{i=1}^{N}x_t^i$. Define the projection error $$e_t^i=\Pi_{\mathcal{X}}\big[\sum_{j=1}^{N}a_{ij}\big(x_{t-1}^j-\eta\widetilde{g}_{t-1}^j(x_{t-1}^j)\big)\big]-\sum_{j=1}^{N}a_{ij}\big(x_{t-1}^j-\eta\widetilde{g}_{t-1}^j(x_{t-1}^j)\big),$$ then $x_t^i=\sum_{j=1}^{N}a_{ij}\big(x_{t-1}^j-\eta\widetilde{g}_{t-1}^j(x_{t-1}^j)\big)+e_t^i$. Based on (\ref{4}), it satisfies that
\begin{align}
\nonumber x_t^i=&\sum_{j=1}^{N}\Big[\Phi(t)\Big]_{ij}x_0^j-\eta\sum_{j=1}^{N}\sum_{\tau=0}^{t-1}\Big[\Phi(t-\tau)\Big]_{ij}\widetilde{g}_{\tau}^j(x_\tau^j)\\
&+\sum_{j=1}^{N}\sum_{\tau=1}^{t}\Big[\Phi(t-\tau)\Big]_{ij}e_\tau^j.
\end{align}
By taking average of both side of above equality, we have that
\begin{align}
\nonumber\bar{x}_t=&\frac{1}{N}\sum_{j=1}^{N}x_0^j-\frac{\eta}{N}\sum_{j=1}^{N}\sum_{\tau=0}^{t-1}\widetilde{g}_{\tau}^j(x_\tau^j)+\frac{1}{N}\sum_{j=1}^{N}\sum_{\tau=1}^{t}e_\tau^j.
\end{align}
To give the bound of projection error $e_t^i$, we use the property of projection and it follows that
\begin{align}
\nonumber\|e_t^i\|^2=&\Big\|\Pi_{\mathcal{X}}\big[\sum_{j=1}^{N}a_{ij}\big(x_{t-1}^j-\eta\widetilde{g}_{t-1}^j(x_{t-1}^j)\big)\big]-\sum_{j=1}^{N}a_{ij}\big(x_{t-1}^j-\eta\widetilde{g}_{t-1}^j(x_{t-1}^j)\big)\Big\|^2\\
\nonumber\leq&2\Big\|\Pi_{\mathcal{X}}\big[\sum_{j=1}^{N}a_{ij}\big(x_{t-1}^j-\eta\widetilde{g}_{t-1}^j(x_{t-1}^j)\big)\big]-\sum_{j=1}^{N}a_{ij}x_{t-1}^j\Big\|^2+2\|\eta\sum_{j=1}^{N}a_{ij}\widetilde{g}_{t-1}^i(x_{t-1}^i)\|^2\\
\nonumber\leq&4\eta^2\sum_{j=1}^{N}\|\widetilde{g}_{t-1}^i(x_{t-1}^i)\|^2.
\end{align}
Combining above inequalities and conclusion of Lemma 1, the term $\|x_t^i-\bar{x}_t\|^2$ can be rewritten as
\begin{align}\label{11}
\nonumber&\|x_t^i-\bar{x}_t\|^2\\
\nonumber=&\Big\|\sum_{j=1}^{N}\Big(\big[\Phi(t)\big]_{ij}-\frac{1}{N}\Big)x_0^j-\eta\sum_{j=1}^{N}\sum_{\tau=0}^{t-1}\Big(\big[\Phi(t-\tau)\big]_{ij}-\frac{1}{N}\Big)\widetilde{g}_{\tau}^j(x_\tau^j)\\
\nonumber&+\sum_{j=1}^{N}\sum_{\tau=1}^{t}\Big(\big[\Phi(t-\tau)\big]_{ij}-\frac{1}{N}\Big)e_\tau^j\Big\|^2\\
\nonumber\leq&3N\gamma^{2(t-1)}\sum_{j=1}^{N}\|x_0^j\|^2+3N\eta^2\sum_{\tau=0}^{t-1}\gamma^{2(t-1-\tau)}\sum_{j=1}^{N}\|\widetilde{g}_{\tau}^j(x_\tau^j)\|^2\\
\nonumber&+3N\eta^2\sum_{\tau=1}^{t}\gamma^{2(t-1-\tau)}\sum_{j=1}^{N}\|e_\tau^j\|^2\\
\nonumber\leq&3N\gamma^{2(t-1)}\sum_{j=1}^{N}\|x_0^j\|^2+15N\eta^2\sum_{\tau=0}^{t-1}\gamma^{2(t-1-\tau)}\sum_{j=1}^{N}\|\widetilde{g}_{\tau}^j(x_\tau^j)\|^2.\\
\end{align}
Then we sum up (\ref{11}) on both sides from $i=1,2,\ldots,N$ and $t=0,1,\ldots,T-1$, it follows that
\begin{align*}
\nonumber&\sum_{t=0}^{T-1}\sum_{i=1}^{N}\|x_t^i-\bar{x}_t\|^2\\
\nonumber\leq&3N^2\sum_{t=0}^{T-1}\gamma^{2(t-1)}\sum_{j=1}^{N}\|x_0^j\|^2+15N^2\eta^2\sum_{t=1}^{T-1}\sum_{\tau=0}^{t-1}\gamma^{2(t-1-\tau)}\sum_{j=1}^{N}\|\widetilde{g}_{\tau}^j(x_\tau^j)\|^2\\
\leq&\frac{3N^3r_u^2}{1-\gamma^2}+\frac{15N^2\eta^2}{1-\gamma^2}\sum_{t=0}^{T-1}\sum_{j=1}^{N}\|\widetilde{g}_{t}^j(x_{t}^j)\|^2.
\end{align*}
where the last inequality holds according to the conclusion that $$\begin{aligned}&\sum_{t=1}^{T-1}\sum_{\tau=0}^{t-1}\gamma^{2(t-1-\tau)}\sum_{j=1}^{n}\|\widetilde{g}_{\tau}^j(x_\tau^j)\|^2\\
=&\sum_{\tau=0}^{T-2}\sum_{t=\tau+1}^{T-1}\gamma^{2(t-2-\tau)}\sum_{j=1}^{n}\|\widetilde{g}_{\tau}^j(x_\tau^j)\|^2\\
\leq&\frac{1}{1-\gamma^2}\sum_{\tau=0}^{T-2}\sum_{j=1}^{N}\|\widetilde{g}_{\tau}^j(x_\tau^j)\|^2.\end{aligned}$$
So we have that 
\begin{align}\label{12}
\nonumber&\sum_{t=0}^{T-1}\sum_{i=1}^{N}\mathbb{E}[\|x_{t+1}^i-x_t^i\|^2]\\
\nonumber\leq&2\sum_{t=0}^{T-1}\sum_{i=1}^{N}\eta^2\mathbb{E}[\|\widetilde{g}_t^i(x_t^i)\|^2]+4\sum_{t=0}^{T-1}\sum_{i=1}^{N}\mathbb{E}\big[\|x_t^i-\bar{x}_t\|^2\big]\\
\leq&\frac{12N^3r_u^2}{1-\gamma^2}+(2\eta^2+\frac{60N^2\eta^2}{1-\gamma^2})\sum_{t=0}^{T-1}\sum_{j=1}^{N}\mathbb{E}\big[\|\widetilde{g}_{t}^j(x_{t}^j)\|^2\big].
\end{align}
To give the bound of (\ref{8}), we have to bound $\sum_{t=1}^{T}\sum_{i=1}^{N}\mathbb{E}[\|\widetilde{g}_t^i(x_t^i)\|^2]$. Substitute (\ref{12}) into (\ref{8}) and note $\alpha=2+\frac{60N^2} {1-\gamma^2}$, we have that
\begin{align}
\nonumber\sum_{t=1}^{T}\sum_{i=1}^{N}\mathbb{E}[\|\widetilde{g}_t^i(x_t^i)\|^2]\leq&\frac{3d^2L_0^2}{\delta^2}\sum_{t=1}^{T}\sum_{i=1}^{N}\mathbb{E}\big[\|x_t^i-x_{t-1}^i\|^2\big]+12d^2L_0^2NT+\frac{3d^2}{\delta^2}\Theta_T^2\\
\nonumber\leq&\frac{36d^2L_0^2 N^3r_u^2}{(1-\gamma^2)\delta^2}+\frac{3\alpha d^2L_0^2\eta^2}{\delta^2}\sum_{t=0}^{T-1}\sum_{i=1}^{N}\|\widetilde{g}_{t}^i(x_{t}^i)\|^2\\
\nonumber&+12d^2L_0^2NT+\frac{3d^2}{\delta^2}\Theta_T^2.
\end{align}
Then, add $\frac{3\alpha d^2L_0^2\eta^2}{\delta^2}\sum_{i=1}^{N}\|\widetilde{g}_{T}^i(x_{T}^i)\|^2$ on the right side and rearrange. Let $\beta=\frac{3\alpha d^2L_0^2\eta^2}{\delta^2}$, it clear that $\beta<1$ from $\eta=\frac{1}{\sqrt{3\alpha}dL_0 T^{\frac{2}{3}}}$ and it follows that
\begin{align}\label{13}
\nonumber&\sum_{t=1}^{T}\sum_{i=1}^{N}\mathbb{E}[\|\widetilde{g}_t^i(x_t^i)\|^2]\\
\leq&\frac{36d^2L_0^2N^3r_u^2}{(1-\gamma^2)(1-\beta)\delta^2}+\frac{12d^2L_0^2NT}{1-\beta}+\frac{3d^2}{(1-\beta)\delta^2}\Theta_T^2+\frac{1}{1-\beta}\sum_{i=1}^{N}\mathbb{E}[\|\widetilde{g}_0^i(x_0^i)\|^2].
\end{align}
Substituting (\ref{8}) and (\ref{13}) into (\ref{6}), it yields that
\begin{align}\label{14}
\nonumber R_g^T\leq&\frac{\eta}{2}\Big[\frac{36d^2L_0^2N^3r_u^2}{(1-\gamma^2)(1-\beta)\delta^2}+\frac{12d^2L_0^2NT}{1-\beta}+\frac{3d^2}{(1-\beta)\delta^2}\Theta_T^2\\
&+\frac{1}{1-\beta}\sum_{i=1}^{N}\mathbb{E}[\|\widetilde{g}_0^i(x_0^i)\|^2]\Big]+\frac{2}{\eta}Nr_u^2+2\delta L_0NT.
\end{align}
Set the parameter $\eta=\frac{1}{\sqrt{3\alpha}dL_0 T^{\frac{2}{3}}}$ and $\delta=\frac{2}{T^{\frac{1}{3}}}$, the regret of Algorithm 1 satisfies
$$R_{g}^T\leq\mathcal{O}\Big(\max\big\{T^{\frac{2}{3}},\Theta_T^2\big\}\Big).$$
\end{proof}
Similar to the analysis above, we consider the distributed online optimization problem where the objective function is smooth with smoothness parameter $L_1$. The following theorem gives the regret bound of Algorithm 1 for smooth and convex objective function.
\begin{theorem}
Suppose Assumption 1 and 2 hold. If $f_{t}^i\in C^{1,1}$ is smooth with smoothness constant $L_1$ for all $t$ and $i$, run Algorithm 1 with $\eta=\frac{1}{\sqrt{3\alpha}dL_0 T^{\frac{1}{2}}}$ and $\delta=\frac{2}{T^{\frac{1}{4}}}$. The regret bound satisfies
$$R_{g}^T\leq\mathcal{O}\Big(\max\big\{T^{\frac{1}{2}},\Theta_T^2\big\}\Big).$$
\end{theorem}
\begin{proof}
According to the conclusion of Lemma 2, we know that $|f_{\delta,t}^i(x)-f_t^i(x)|\leq\delta^2L_1$. Following the same proof logic of Theorem 1, we simply substitute the term $2\delta L_0NT$ in (\ref{14}) with $2\delta^2L_1NT$. Then we can obtain that
\begin{align}\label{15}
\nonumber R_g^T\leq&\frac{\eta}{2}\Big[\frac{36d^2L_0^2N^3r_u^2}{(1-\gamma^2)(1-\beta)\delta^2}+\frac{12d^2L_0^2NT}{1-\beta}+\frac{3d^2}{(1-\beta)\delta^2}\Theta_T^2\\
&+\frac{1}{1-\beta}\sum_{i=1}^{N}\mathbb{E}[\|\widetilde{g}_0^i(x_0^i)\|^2]\Big]+\frac{2}{\eta}Nr_u^2+2\delta^2L_1NT.
\end{align}
Set the parameter $\eta=\frac{1}{\sqrt{3\alpha}dL_0 T^{\frac{1}{2}}}$ and $\delta=\frac{2}{T^{\frac{1}{4}}}$, the regret of Algorithm 1 satisfies
$$R_{g}^T\leq\mathcal{O}\Big(\max\big\{T^{\frac{1}{2}},\Theta_T^2\big\}\Big).$$
\end{proof}

\section{Distributed ORF Algorithm for Non-Convex Online Optimization}
In this section, we consider the distributed online non-convex optimization problems and give the following projected update rule with ORF:
\begin{equation}\label{16}
  x_{t+1}^i=\Pi_{\mathcal{X}}\big[\sum_{j=1}^{N}a_{ij}x_t^j-\eta\widetilde{g}_t^i(x_t^i)\big].
\end{equation}

First, we consider the case where the objective functions $\{f_{t}^i\}$ are non-convex and Lipschitz continuous. For non-convex optimization problems, it may be difficult to find the global optimal point of the objective function and the accumulation of gradient is always used to redefine the regret. In practical problems, the objective function $f_t^i$ may be not necessarily differentiable. Here, we use the sum of gradient of the smoothed function to define the regret as follows, $$R_{g,\delta}^{T}:=\sum_{t=0}^{T-1}\sum_{i=1}^{N}\mathbb{E}[\|\nabla f_{\delta,t}^i(x_{t}^i)\|^{2}].$$ 

Similar to the analysis of ZO in \cite{ref16} for static nonsmooth optimization problems, we have to make the smoothed function $f_{\delta,t}^i$ is as close as possible to the original function $f_{t}^i$. Based on Lemma 2, we set  $\delta\leq(L_0)^{-1}\epsilon_f$ to make the difference of $f_{\delta,t}^i$ and $f_{t}^i$ is less than a small positive scaler $\epsilon_f$. Then, we define the increasing rate of smoothed function as follows:$$
\theta_{\delta,t}^i=\sup\limits_{x\in\mathcal{X},t=1,2,\ldots,T}\big|f_{\delta,t+1}^i(x)-f_{\delta,t}^i(x)\big|,$$
$$\Theta_{T,\delta}=\sum_{t=1}^{T}\sum_{i=1}^{N}\theta_{\delta,t}^i.$$
Since the difference between smoothed function and original function is small enough, we further assume that $\Theta_{T,\delta}=\mathcal{O}(\Theta_T)$, which means the accumulated increasing rate of smoothed function is same as original function. Now, we give the following algorithm and theorem of regret bound for non-convex optimization problem.
\clearpage
\begin{algorithm}
\caption{Distributed ORF Algorithm for Online Non-Convex Optimization}
{\textbf{Initialization}:}
{Initial values of $x^1_0, x^2_0,$
    $ \cdots, x^N_0$, number of iterations T,
    and appropriate value of $\eta$ and $\delta$.}

\textbf{For }$t=0$ to $T$, $i=1$ to $N$ 
   
   \qquad Let $u_t^i$ uniformly sampled in $\mathbb{S}^{d}$ and compute the ORF estimator $\widetilde{g}_{t}^i(x_{t}^i)$ by (\ref{2}).
   
   \qquad Update $x^i_{t+1}$ for all agents $i$ by (\ref{16}).
   
\textbf{end for}
\end{algorithm}

\begin{theorem}
Suppose Assumption 1 and 2 hold. If $f_{t}^i\in C^{0,0}$ with Lipschitz constant $L_0$ for all $t$ and $i$, run Algorithm 2 with $\eta=\frac{1}{\sqrt{3\alpha}dL_0T^{\frac{1}{4}}}$ and $\delta=\frac{\epsilon_f}{L_{0}}$. The regret bound satisfies
$$R_{g,\delta}^T\leq\mathcal{O}\Big(\max\big\{T^{\frac{1}{4}}\Theta_T,T^{\frac{3}{4}},\frac{\Theta_T^2}{T^{\frac{1}{4}}}\big\}\Big).$$

\end{theorem}
\begin{proof}
According to Lemma 3 and $f_t^i(x)\in C^{0,0}$, the gradient of smoothed function $f_{\delta,t}^i(x)$ is $L_{\delta}$-Lipschitz, where $L_{\delta}=\frac d\delta L_0.$ In addition, from Lemma 1.2.3 in \cite{ref26}, we can obtain that
\begin{align}\label{17}
\nonumber \mathbb{E}\big[\sum_{t=0}^{T-1}\sum_{i=1}^{N}f_{\delta,t}^i(x_{t+1}^i)\big]\leq&\mathbb{E}\Big[\sum_{t=0}^{T-1}\sum_{i=1}^{N}\big[f_{\delta,t}^i(x_t^i)+\langle\nabla f_{\delta,t}^i(x_t^i),x_{t+1}^i-x_t^i\rangle\\
&+\frac{L_{\delta}}{2}\|x_{t+1}^i-x_t^i\|^2\big]\Big].\end{align}

\noindent From Lemma 4, we have that $\mathbb{E}_{u_{t}^i}[\tilde{g}_{t}^i(x_{t}^i)]=\nabla f_{\delta,t}^i(x_{t}^i).$ Since the radium of constraint set is bounded by $r_u$, we redefine projection error as $$e_t^i=\Pi_{\mathcal{X}}\big[\sum_{j=1}^{N}a_{ij}x_{t-1}^j-\eta\widetilde{g}_{t-1}^i(x_{t-1}^i)\big]-\big(\sum_{j=1}^{N}a_{ij}x_{t-1}^j-\eta\widetilde{g}_{t-1}^i(x_{t-1}^i)\big),$$ it follows that 
\begin{align}\label{18}
\nonumber&\sum_{t=0}^{T-1}\sum_{i=1}^{N}\mathbb{E}[\langle\nabla f_{\delta,t}^i(x_t^i),x_{t+1}^i-x_t^i\rangle]\\
\nonumber=&\sum_{t=0}^{T-1}\sum_{i=1}^{N}\mathbb{E}[\langle\nabla \nonumber f_{\delta,t}^i(x_t^i),e_{t+1}^i+\sum_{j=1}^{N}a_{ij}x_t^j-\eta\widetilde{g}_t^i(x_t^i)-x_t^i\rangle]\\
\nonumber=&-\sum_{t=0}^{T-1}\sum_{i=1}^{N}\eta\mathbb{E}[\|\nabla f_{\delta,t}^i(x_{t}^i)\|^{2}]+\sum_{t=0}^{T-1}\sum_{i=1}^{N}\mathbb{E}[\langle\nabla \nonumber f_{\delta,t}^i(x_t^i),\sum_{j=1}^{N}a_{ij}x_t^j-x_t^i+e_{t+1}^i\rangle]\\
\nonumber\leq&-\eta\sum_{t=0}^{T-1}\sum_{i=1}^{N}\mathbb{E}[\|\nabla f_{\delta,t}^i(x_{t}^i)\|^{2}]+2\sum_{t=0}^{T-1}\sum_{i=1}^{N} r_u\mathbb{E}[\|\widetilde{g}_t^i(x_t^i)\|]\\
&+\sum_{t=0}^{T-1}\sum_{i=1}^{N}\mathbb{E}[\langle\nabla f_{\delta,t}^i(x_t^i),e_{t+1}^i\rangle].\end{align}
From the definition of $e_t^i$ and (\ref{3}), we have that
$$\begin{aligned}
\nonumber\|e_{t+1}^i\|=&\Big\|\Pi_{\mathcal{X}}\big[\sum_{j=1}^{N}a_{ij}x_{t}^j-\eta\widetilde{g}_{t}^i(x_{t}^i)\big]-\big(\sum_{j=1}^{N}a_{ij}x_{t}^j-\eta\widetilde{g}_{t}^i(x_{t}^i)\big)\Big\|\\
\nonumber\leq&\Big\|\Pi_{\mathcal{X}}\big[\sum_{j=1}^{N}a_{ij}x_{t}^j-\eta\widetilde{g}_{t}^i(x_{t}^i)\big]-\sum_{j=1}^{N}a_{ij}x_{t}^j\Big\|+\|\eta\widetilde{g}_{t}^i(x_{t}^i)\|\\
\nonumber\leq&2\eta\|\widetilde{g}_{t}^i(x_{t}^i)\|\\
\nonumber\leq&2\eta\big[\frac{2\sqrt{3}dL_0r_u}{\delta}+2\sqrt{3}dL_0+\frac{\sqrt{3}d}{\delta}\theta_{i,t}^2\big].
\end{aligned}$$
Since $\Theta_T/T\rightarrow0$ as $T\rightarrow\infty$, $\theta_{i,t}$ is clearly bounded by a positive constant. Then, we let $\overline{\theta}=\max\limits_{i=1,2,\ldots,N,t=1,2,\ldots,T}\{\theta_{i,t}\}$ and note $B=\frac{2\sqrt{3}dL_0r_u}{\delta}+2\sqrt{3}dL_0+\frac{2\sqrt{3}d}{\delta}\overline{\theta}$ and have that $\|e_{t+1}^i\|\leq2\eta B$. Combine (\ref{18}) and the inequality above, we have that
$$\begin{aligned}
&\sum_{t=0}^{T-1}\sum_{i=1}^{N}\mathbb{E}[\langle\nabla f_{\delta,t}^i(x_t^i),x_{t+1}^i-x_t^i\rangle]\\
\leq&-\eta\sum_{t=0}^{T-1}\sum_{i=1}^{N}\mathbb{E}[\|\nabla f_{\delta,t}^i(x_{t}^i)\|^{2}]+2(r_u+\eta B)\sum_{t=0}^{T-1}\sum_{i=1}^{N} \mathbb{E}[\|\widetilde{g}_t^i(x_ t^i)\|].
\end{aligned}$$
Substitute above inequality into (\ref{17}) and rearrange, it holds that
\begin{align}\label{19}
\nonumber&\sum_{t=0}^{T-1}\sum_{i=1}^{N}\mathbb{E}[\|\nabla f_{\delta,t}^i(x_{t}^i)\|^{2}]\\
\nonumber\leq&\frac{1}{\eta}\sum_{t=0}^{T-1}\sum_{i=1}^{N}\mathbb{E}\Big[f_{\delta,t}^i(x_t^i)-f_{\delta,t}^i(x_{t+1}^i)+2(r_u+\eta B)\mathbb{E}[\|\widetilde{g}_t^i(x_t^i)\|]\\
&+\frac{L_{\delta}}{2}\|x_{t+1}^i-x_t^i\|^2\Big].
\end{align}
In addition, based on (\ref{13}) and the conclusion that $\mathbb{E}[\mathbf{x}^2]\geq(\mathbb{E}\mathbf{x})^2$ and $\sum{a_i^2}\leq(\sum{a_i})^2$, we can also give the bound of $\sum_{t=0}^{T-1}\sum_{i=1}^{N}\mathbb{E}[\|\widetilde{g}_t^i(x_t^i)\|]$ as
\begin{align}\label{20}
\nonumber&\sum_{t=0}^{T-1}\sum_{i=1}^{N}\mathbb{E}[\|\widetilde{g}_t^i(x_t^i)\|]\leq\sum_{t=0}^{T}\sum_{i=1}^{N}\mathbb{E}[\|\widetilde{g}_t^i(x_t^i)\|]\\
\leq&\frac{6dL_0 N^{\frac{3}{2}}r_u}{\sqrt{(1-\gamma^2)(1-\beta)}\delta}+\frac{2\sqrt{3NT}dL_0}{\sqrt{1-\beta}}+\frac{\sqrt{3}d}{\sqrt{1-\beta}\delta}\Theta_T+\sqrt{G}.
\end{align}
where $G=\frac{2-\beta}{1-\beta}\sum_{i=1}^{N}\mathbb{E}[\|\widetilde{g}_0^i(x_0^i)\|^2]$. Combining (\ref{12}) and (\ref{13}), we can give the bound of $\sum_{t=0}^{T-1}\sum_{i=1}^{N}\|x_{t+1}^i-x_t^i\|^2$ as
\begin{align}\label{21}
\nonumber&\sum_{t=0}^{T-1}\sum_{i=1}^{N}\mathbb{E}[\|x_{t+1}^i-x_t^i\|^2]\\
\nonumber\leq&\frac{12N^3r_u^2}{1-\gamma^2}+\alpha\eta^2\sum_{t=0}^{T-1}\sum_{i=1}^{N}\mathbb{E}[\|\widetilde{g}_t^i(x_t^i)\|^2]\\
\nonumber\leq&\frac{12N^3r_u^2}{1-\gamma^2}+\frac{36d^2L_0^2N^3r_u^2\alpha\eta^2}{(1-\gamma^2)(1-\beta)\delta^2}+\frac{12d^2L_0^2NT\alpha\eta^2}{1-\beta}\\
&+\frac{3d^2\alpha\eta^2}{(1-\beta)\delta^2}\Theta_T^2+\alpha\eta^2G.
\end{align}
Similar to the proof of Theorem 1, we substitute (\ref{20}) and (\ref{21}) into (\ref{19}) and it yields that
\begin{align}\label{22}
\nonumber R_{g,\delta}^T=&\sum_{t=0}^{T-1}\sum_{i=1}^{N}\mathbb{E}[\|\nabla f_{\delta,t}^i(x_{t}^i)\|^{2}]\\
\nonumber\leq&\frac{1}{\eta}\sum_{t=0}^{T-1}\sum_{i=1}^{N}\big[f_{\delta,t}^i(x_t^i)+f_{\delta,t+1}^i(x_{t+1}^i)-f_{\delta,t+1}^i(x_{t+1}^i)-f_{\delta,t}^i(x_{t+1}^i)\big]\\
\nonumber&+2(\frac{r_u}{\eta}+B)\sum_{t=0}^{T-1}\sum_{i=1}^{N}\mathbb{E}[\|\widetilde{g}_t^i(x_t^i)\|]+\frac{dL_{0}}{2\delta\eta}\mathbb{E}\big[\sum_{t=0}^{T-1}\sum_{i=1}^{N}\|x_{t+1}^i-x_t^i\|^2\big]\\
\nonumber\leq&\frac{1}{\eta}\sum_{i=1}^{N}\big[f_{\delta,0}^i(x_0^i)-f_{\delta,T}^i(x_{T}^i)\big]+\frac{1}{\eta}\Theta_T+\frac{dL_{0}}{2\delta}\frac{24N^3r_u^2}{1-\gamma^2}\\
\nonumber&+\frac{dL_{0}}{2\delta}\frac{72d^2L_0^2N^3r_u^2\alpha\eta}{(1-\gamma^2)(1-\beta)\delta^2}+\frac{dL_{0}}{2\delta}\frac{12d^2L_0^2NT\alpha\eta}{1-\beta}\\
&+\frac{dL_{0}}{2\delta}\frac{3d^2\alpha\eta}{(1-\beta)\delta^2}\Theta_T^2+\frac{dL_{0}\alpha\eta G}{2\delta}+2(\frac{r_u}{\eta}+B)\sum_{t=0}^{T-1}\sum_{i=1}^{N}\mathbb{E}[\|\widetilde{g}_t^i(x_t^i)\|].
\end{align}
Since $2(\frac{r_u}{\eta}+B)\sum_{t=0}^{T-1}\sum_{i=1}^{N}\mathbb{E}[\|\widetilde{g}_t^i(x_t^i)\|=\mathcal{O}(\max\{\frac{\Theta_T}{\eta},\frac{\sqrt{T}}{\eta}\})$, the regret bound of Algorithm 2 is bounded as 
\begin{align*}
R_{g,\delta}^T\leq\mathcal{O}\Big(\max\big\{\frac{\Theta_T}{\eta},\frac{\sqrt{T}}{\eta},\eta T,\eta\Theta_T^2\big\}\Big).
\end{align*}
Through simple computation, the regret of our algorithm can achieve $O(T^{\frac{3}{4}})$ when $\eta=\frac{1}{\sqrt{3\alpha}dL_0T^{\frac{1}{4}}}$ and the increasing rate of $\Theta_T$ is no more than $O(\sqrt{T})$. In addition, to satisfy that $|f_t^i(x)-f_{\delta,t}^i(x)|\leq\epsilon_f$, we choose $\delta=\frac{\epsilon_f}{L_{0}}$ and the proof is complete.
\end{proof}

Next, we consider the distributed non-convex online optimization problem when the objective functions are smoothed and define the regret $$R_g^T=\sum_{t=0}^{T-1}\sum_{i=1}^{N}\mathbb{E}[\|\nabla f_t^i(x_t)\|^2].$$ Similar to the above result, we give the following theorem.

\begin{theorem}
Suppose Assumption 1 and 2 hold. If $f_{t}^i\in C^{1,1}$ is $L_0$-Lipschitz and smoothed with smoothness constant $L_1$ for all $t$ and $i$, run Algorithm 2 with $\eta=\frac{1}{\sqrt{3\alpha}dL_0T^{\frac{1}{4}}}$ and $\delta=\frac{d}{T^{\frac{1}{8}}}$. The regret bound satisfies
$$R_{g}^T\leq\mathcal{O}\Big(\max\big\{T^{\frac{3}{4}},T^{\frac{1}{4}}\Theta_T,T^{\frac{3}{8}}\Theta_T,\Theta_T^2\big\}\Big).$$
\end{theorem}
\begin{proof}
Since the proof is similar to the proof of Theorem 3, we just simply replace $L_{\delta}$ with $L_1$ in (\ref{21}). So we have that

\begin{align}\label{23}
\nonumber &\sum_{t=0}^{T-1}\sum_{i=1}^{N}\mathbb{E}[\|\nabla f_{\delta,t}^i(x_{t}^i)\|^{2}]\\
\nonumber\leq&\frac{1}{\eta}\sum_{i=1}^{N}\big[f_{\delta,0}^i(x_0^i)-f_{\delta,T}^i(x_{T}^i)\big]+\frac{1}{\eta}\Theta_T+\frac{24L_1N^3r_u^2}{2(1-\gamma^2)}\\
\nonumber&+\frac{36d^2L_1L_0^2N^3r_u^2\alpha\eta}{(1-\gamma^2)(1-\beta)\delta^2}+\frac{6d^2L_1L_0^2NT\alpha\eta}{1-\beta}\\
&+\frac{3d^2L_1\alpha\eta}{2(1-\beta)\delta^2}\Theta_T^2+\frac{\alpha\eta L_1G}{2}+2(\frac{r_u}{\eta}+B)\sum_{t=0}^{T-1}\sum_{i=1}^{N}\mathbb{E}[\|\widetilde{g}_t^i(x_t^i)\|].
\end{align}

Since $f_t^i\in C^{1,1}$, we know that $\|\nabla f_{\delta,t}^i(x)-\nabla f_t^i(x)\|\leq dL_1\delta$ ac     cording to Lemma 2. Furthermore, we can obtain that
\begin{align}\label{24}
&\sum_{t=0}^{T-1}\sum_{i=1}^{N}\mathbb{E}[\|\nabla f_t^i(x_t^i)\|^2]\nonumber\\
=&\sum_{t=0}^{T-1}\sum_{i=1}^{N}\mathbb{E}[\|\nabla f_t^i(x_t^i)-\nabla f_{\delta,t}^i(x_t^i)+\nabla f_{\delta,t}^i(x_t^i)\|^2]\nonumber\\
\leq&\ 2\sum_{t=0}^{T-1}\sum_{i=1}^{N}\mathbb{E}[\|\nabla f_t^i(x_{t}^i)-\nabla f_{\delta,t}^i(x_{t}^i)\|^{2}]\nonumber\\
&+2\sum_{t=0}^{T-1}\sum_{i=1}^{N}\mathbb{E}[\|\nabla f_{\delta,t}^i(x_{t}^i)\|^{2}]\nonumber\\
\leq&\ 2d^2L_1^2\delta^2 NT+2\sum_{t=0}^{T-1}\sum_{i=1}^{N}\mathbb{E}[\|\nabla f_{\delta,t}^i(x_{t}^i)\|^{2}].\end{align}
Then we set $\eta=\frac{1}{\sqrt{3\alpha}dL_0T^{\frac{1}{4}}}$ and $\delta=\frac{d}{T^{\frac{1}{8}}}$ to satisfies $\beta\leq1$. Combining (\ref{20}) (\ref{23}) and (\ref{24}), we can conclude that
$$\begin{aligned}
R_{g}^T\leq\mathcal{O}\Big(\max\big\{T^{\frac{3}{4}},T^{\frac{1}{4}}\Theta_T,T^{\frac{3}{8}}\Theta_T,\Theta_T^2\big\}\Big).\end{aligned}$$
When increasing rate of $\Theta_T$ is no more than $O(T^{\frac{3}{8}})$, our algorithm can achieve $O(T^{\frac{3}{4}})$ regret bound.
\end{proof}

\section{Conclusion}
In this paper, we extend the ORF estimator to the distributed online optimization problems. It is assumed that the communication graph is undirected and adjacency matrix is double-stochastic. Furthermore, we consider a sequence of non-stationary objective functions, where the decision of each agent at time $t$ is influenced by the decisions of others. Then, we design two algorithms based on  one-point residual feedback estimator for convex and non-convex optimization problems and analyze the regret performance. It is shown that the algorithms can achieve a sublinear regret bound. Both theoretical analysis and numerical examples substantiate that the ORF estimator exhibits a lower regret bound and variance compared to traditional one-point estimators, thereby enhancing the convergence rate.
\section{Acknowledgement}
This work is supported by the National Natural Science Foundation of China (62473009).

\bibliographystyle{elsarticle-num} 
\bibliography{reference.bib}





\end{document}